\documentclass [12pt]{article}
\usepackage[english]{babel}
\usepackage{graphicx}
\usepackage{amssymb,amsmath,amsthm,amscd}
\newtheorem{theorem}{Theorem}
\newtheorem{lemma}{Lemma}
\newtheorem{proposition}{Proposition}

\newtheoremstyle{neosn}{0.5\topsep}{0.5\topsep}{\rm}{}{\sc}{.}{ }{\thmname{#1}\thmnumber{ #2}\thmnote{ {\mdseries#3}}}
\theoremstyle{neosn}
\newtheorem{definition}{Definition}

\newcommand{\Rad}{\,\mathrm{Rad}\,}
\newcommand{\ad}{\,\mathrm{ad}\,}
\newcommand{\kerr}{\,\mathrm{ker}\,}
\newcommand{\GL}{\,\mathrm{GL}\,}
\newcommand{\SL}{\,\mathrm{SL}\,}
\newcommand{\PSL}{\,\mathrm{PSL}\,}

\newcommand{\Sp}{\,\mathrm{Sp}\,}
\newcommand{\calL}{{\mathcal L}}
\newcommand{\calH}{{\mathcal H}}
\newcommand{\Spin}{\,\mathrm{Spin}\,}

\newcommand{\SO}{\,\mathrm{SO}\,}
\newcommand{\Aut}{\,\mathrm{Aut}\,}

\begin{document}

\setcounter{MaxMatrixCols}{20}

\begin{center}

{\Large {\bf Automorphisms of a Chevalley group of type $\mathbf G_2$ \\

\bigskip

over a commutative ring $R$ with $1/3$  \\

\bigskip

generated by the all  invertible elements and $2R$ }}

\bigskip
\bigskip

{\large \bf E.~I.~Bunina, M.~A.~Vladykina}

\end{center}

\bigskip

\begin{center}

{\bf  Abstract.}

In this paper we prove that every automorphism of  a Chevalley group with the root system  $\mathbf G_2$ over a commutative ring $R$ with  $1/3$, generated by  all its invertible elements and  the ideal $2R$ is a composition of ring  and inner automorphisms. 

\end{center}

\bigskip

\section*{Introduction}\leavevmode

Study of automorphisms of classical groups was started by the work of Schreier and van der Warden~\cite{Schreier}  in 1928. They described all automorphisms of the group ${\PSL}_{n}$ ($n \geqslant 3$) over an arbitrary field.

Diedonne~\cite{Diedonne1} (1951) and Rickart~\cite{Rickart}  (1950) introduced the involution method, and with the help of this method described automorphisms of the group  ${\GL}_n$ ($n \geqslant 3$) over a skewfield.

The first step in construction of the automorphism theory over rings, namely, for the group ${\GL}_{n}$ ($n \geqslant 3$) over the ring of integer numbers, was made by  Hua Logen and Reiner~\cite{Hua} (1951), later some papers over commutative integral domains appeared.
The methods of the papers mentioned above were based mostly on studying involutions in the corresponding linear groups.

O'Meara~\cite{O'Meara} in 1976 invented very different (geometrical) method, which did not use involutions. By its aid, O"Meara described automorphisms of the group $\GL_n$ ($n\geqslant 3$) over domains.

In 1982 Petechuk~\cite{Petechuk1}  described automorphisms of the groups ${\GL}_n, {\SL}_n (n \geqslant 4)$ over arbitrary commutative rings. If $n = 3$, automorphisms of given linear groups are not always standard. They are standard either if in a ring $2$ is invertible, or if a ring is a domain, or it is a semisimple ring.

Isomorphisms of the groups ${GL}_{n}(R)$ and ${GL}_{m}(S)$ over arbitrary associative rings with $1/2$ for $n, m \geqslant 3$ were described in 1981 by I.Z.\,Golubchik and A.V.\,Mikhalev~\cite{Golubchik1}  and independently by E.I.\,Zelmanov~\cite{Zelmanov}. In 1997 I.Z\, Golubchik~\cite{Golubchik3} described isomorphisms between these groups for $n,m \geqslant 4$, but over arbitrary associative rings with~$1$.

In  50-th  years of the previous century Chevalley~\cite{Chevalley1}, Steinberg~\cite{Steinberg1}  and others introduced the concept of Chevalley groups
over commutative rings, which includes classical linear groups (special linear ${\SL}$, special orthogonal ${\SO}$, symplectic ${\Sp}$, spinor ${\Spin}$, and also projective groups connected with them) over commutative rings.

Clear that isomorphisms and automorphisms of Chevalley groups were also studied intensively.
The description of isomorphisms of Chevalley groups over fields was obtained by Steinberg  for the finite case~\cite{Steinberg2} and by Humphreys for the infinite one~\cite{Humphreys1}. Many papers were devoted to description of automorphisms of Chevalley groups over different commutative rings, we can mention here the papers of Borel--Tits~\cite{Borel2}, Carter--Chen Yu~\cite{Carter2}, Chen Yu~\cite{Chen1}, E. Abe~\cite{Abe1}, A.\,Klyachko~\cite{Klyachko}.

In the paper~\cite{Bunina1}  Bunina proved that automorphisms of adjoint elementary Chevalley groups with root systems $\mathbf A_{l},\mathbf D_{l},\mathbf E_{l},$ $l  \geqslant 2$, over local rings with invertible~$2$ can be represented as the composition of a ring automorphism and an automorphism-conjugation by some matrix from the normalizer of this group in ${\GL}(V)$. In the paper~\cite{Bunina3} according to the results of~\cite{Bunina1} it was shown that every automorphism of an elementary Chevalley group of the described type is standard, i.\,e.,  is represented by the composition of ring, inner, central and graph automorphisms. In the same paper the theorem describing the normalizer of Chevalley groups in their adjoint representation, which also holds for local rings without $1/2$, was obtained.

In the papers~\cite{Bunina2},~\cite{Bunina4},~\cite{Bunina5}, by the same methods it was shown that all automorphisms of Chevalley groups with the root systems $\mathbf F_4, \mathbf  G_2,\mathbf B_l$, $l \geqslant 2$, over local rings with $1/2$ (in the case $\mathbf G_2$ also with $1/3$) are standard. In the paper~\cite{Bunina6} it was proved that all automorphisms of Chevalley groups of types $ \mathbf A_{l}, \mathbf D_{l}, \mathbf E_{l}, l  \geqslant 3$, over local rings without~$1/2$, are standard.

In the paper~\cite{Bunina7} with the help of all previous results and localization method the  automorphisms of adjoint Chevalley groups over arbitrary commutative rings were described, where a root system has rank $>1$, and for $ \mathbf A_2, \mathbf F_4, \mathbf  B_l, \mathbf C_l$ the ring contains~$1/2$, for $\mathbf G_2$ the ring contains $1/2$ and $1/3$. In the recent paper~\cite{Bunina-new} this result was extended to all Chevalley groups (not only adjoint) with the same restrictions for rings.

In  the papers \cite{Bunina8}, \cite{Bunina9} E.I.\,Bunina and  P.A.\,Weryovkin  proved that every automorphism of a Chevalley group of the type $ \mathbf G_2$  over a local ring with $1/3$ and without $1/2$  is standard.

Over fields of characteristic~$2$ and the root systems~$\mathbf F_4$  there exists a non-standard automorphism of the corresponding Chevalley group  (see \cite{Steinberg1}). Over fields of characteristic~$2$ and the root systems $\mathbf B_l$ and $\mathbf C_l$ here exists a non-standard isomorphism  between the corresponding Chevalley groups (also see~\cite{Steinberg1}). In the case of local rings without $1/2$ and the root system $\mathbf A_2$ there is a non-standard automorphism  of corresponding Chevalley group too (see~\cite{Petechuk2} and~\cite{Petechuk3}). For the root system $\mathbf G_2$ non-invertible $3$  is a significant obstacle for standardity of automorphisms, since even for  fields of characteristic~$3$ a non-standard automorphism  exists (see~\cite{Steinberg1}), but as we see in the paper~\cite{Bunina9}, non-invertibility of $2$ does not interfere with standardity of automorphisms  for fields and for local rings.

 So it is a natural goal to continue the result  of standardity of automorphisms of the type $\mathbf G_2$ for  arbitrary commutative rings. At the moment, this has been done for a commutative ring with~$1/3$ generated by its invertible elements and the ideal~$2R$.

\newpage

\section{Definitions and main theorem.}\leavevmode

We fix the root system of type $\mathbf G_{2}$ with the system of simple roots $\Delta(\mathbf G_{2})=\{\alpha_{1}=e_{1}-e_{2}, \alpha_{2}= -2e_{1}+e_{2}+e_{3}\}$.

The system of positive roots :
$\mathbf G_2^{+}= \{\alpha_{1},\alpha_{2},\alpha_{1}+\alpha_{2}=e_{1}, 2\alpha_{1}+\alpha_{2},3\alpha_{1}+\alpha_{2}, 3\alpha_{1}+2\alpha_{2}\} $.

\begin{center}
\includegraphics[width=80mm]{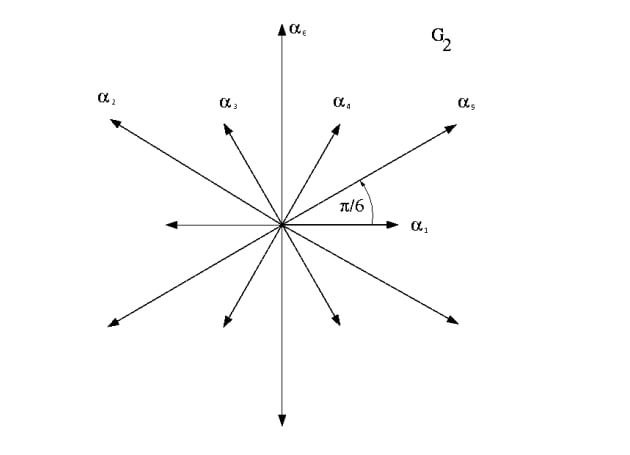}
\end{center}

More details about root systems and their properties can be found in~\cite{Bourbaki1} and~\cite{Humphreys2}.

Suppose now that we have some semisimple complex Lie algebra ${\calL}$ of type $\mathbf G_2$ with Cartan subalgebra ${\calH}$   (detailed information about semisimple Lie algebras can be found in the book~\cite{Humphreys2}).
Then we can choose a basis $\{h_1, . . . ,h_l\} \in {\calH}$ and for every $\alpha \in  \mathbf G_2$  elements $x_{\alpha} \in {\calL}_{\alpha}$ so that $\{h_i;x_{\alpha}\}$ form a basis in ${\calL}$ and for every two elements of this basis their commutator is an integral linear combination of the elements of the same basis.

Let us introduce elementary Chevalley groups (see, for example,\cite{Steinberg1}).
Let ${\calL}$ be a semisimple Lie algebra (over ${\calL}$) with a root system $\mathbf G_2, \, \pi : {\calL} \rightarrow gl_{n}(V)$ be its finitely dimensional faithful representation (of dimension~$n$). If ${\calH}$ is a Cartan subalgebra of~${\calL}$, then a functional $\lambda \in {\calH}^*$  is called a \emph{weight} of a given representation, if there exists a nonzero vector $v \in V$ (called a \emph{weight vector})    such that for any $h \in {\calH} \, \pi(h)v = \lambda(h)v$
in the space $V$ there exists a basis of weight vectors such that all operators $t^k \pi(x_{\alpha})^k/k! \quad k \in \mathbb{N}$  are written as integral (nilpotent) matrices. This basis is called a \emph{Chevalley basis}. An integral matrix also can be considered as a matrix over an arbitrary commutative ring with $1$. Let $R$ be such a ring. Consider matrices $n \times n$ over~$R$, matrices $t^k \pi(x_{\alpha})^k/k!$ for  $\alpha \in \mathbf G_2, k \in \mathbb{N}$  are included in $M_{n}(R)$.

Now consider automorphisms of the free module $\Aut(R^n)$ of the form 
$$
x_{\alpha}(t) :=\exp(tx_{\alpha}) = 1+t\pi(x_{\alpha}) + \dots + t^k \pi(x_{\alpha})^k/k! +\dots
$$ 
Since all matrices $x_{\alpha}$ are nilpotent, this series is finite. Automorphisms $x_{\alpha}(t)$ are called \emph{elementary root elements}. The subgroup $\Aut(R^n)$, generated by $x_{\alpha}(t),  \alpha \in \mathbf G_2$, is called an \emph{elementary adjoint Chevalley group}. 

The action of $x_{\alpha}(t)$ on the Chevalley basis is described in~\cite{Vavilov3}.

All weights of a given representation  generate (by addition)  a lattice  (a free Abelian group, where every $\mathbb{Z}$-basis is also a $\mathbb{C}$-basis in ${\calH}^*$), which is called the \emph{weight lattice} $\Lambda_{\pi}$.

 Elementary Chevalley groups are defined not even by a representation of the Chevalley groups,
but just by its \emph{weight lattice}. More precisely, up to an abstract
isomorphism an elementary Chevalley group is completely defined by a
root system~$\Phi$, a commutative ring~$R$ with~$1$ and a weight
lattice~$\Lambda_\pi$.

Among all lattices we mark  the lattice corresponding to the
adjoint representation, it is generated by all roots (the \emph{root
lattice}~$\Lambda_{ad}$).
  The corresponding (elementary) Chevalley group is called \emph{adjoint}.

Note that for the root system $\mathbf G_2$ there exists only one weight lattice, which is simultaneously simply connected and adjoint, therefore for every ring~$R$ there exists only one Chevalley group of the type   $\mathbf G_2$: $G(R)=G_{\ad}(\mathbf G_2,R)$. 

\bigskip

In elementary Chevalley group it is possible to consider the following important elements:

--- $w_{\alpha}(t) = x_{\alpha}(t)x_{-\alpha}(-t^{-1})x_{\alpha}(t), \; t \in R^{\star}$;

--- $h_{\alpha}(t) = w_{\alpha}(t)w_{\alpha}(1)^{-1}$.

\bigskip

Every elementary Chevalley group satisfies the following relations:

(R1) $\forall \alpha \in \mathbf G_2, \forall t,u \in R,\; x_{\alpha}(t)x_{\alpha}(u) =  x_{\alpha}(t+u); $

(R2) If $ \{\alpha, \beta\}$ are simple roots of the system $\mathbf  G_{2}$, $\alpha$ is long, $\beta$ is short, $t,u\in R$, then
\begin{align*}
\left[ x_{\alpha}(t),x_{\beta}(u) \right] &= x_{\alpha+\beta}(tu)x_{\alpha+3\beta}(-tu^{3})x_{\alpha+2\beta}(-tu^{2})x_{2\alpha+3\beta}(t^{2}u^{3}),\\
\left[ x_{\alpha+\beta}(t),x_{\beta}(u) \right] &= x_{\alpha+2\beta}(2tu)x_{\alpha+3\beta}(-3tu^{2})x_{2\alpha+3\beta}(3t^{2}u^{3});\\
\left[ x_{\alpha}(t),x_{\alpha+3\beta}(u) \right]& = x_{2\alpha+3\beta}(tu),\\
\left[ x_{\alpha+2\beta}(t),x_{\beta}(u) \right] &= x_{\alpha+3\beta}(-3tu),\\
\left[ x_{\alpha+\beta}(t),x_{\alpha+2\beta}(u) \right] &= x_{2\alpha+3\beta}(3tu);
\end{align*}

(R3) $\forall \alpha \in \mathbf G_{2} \, \, \, w_{\alpha} : = w_{\alpha}(1)$;

(R4) $\forall \alpha, \beta \in \mathbf G_{2}, \forall t \in R^{\star} \;  w_{\alpha}h_{\beta}(t)w_{\alpha}^{-1} = h_{w_{\alpha}(\beta)}(t)$;

(R5)  $\forall \alpha, \beta \in \mathbf G_{2} \, \, \, \forall t \in R \;  w_{\alpha}x_{\beta}(t)w_{\alpha}^{-1} = x_{w_{\alpha}(\beta)}(ct);$, where $c = c(\alpha,\beta) = \pm 1$;

(R6)  $\forall \alpha, \beta \in \mathbf G_{2} \, \, \, \forall t \in R^{\star} \; \forall u \in R \;  h_{\alpha}(t)x_{\beta}(u)h_{\alpha}(t)^{-1} = x_{\beta}(t^{\langle \beta, \alpha \rangle} u)$.

\medskip

\bigskip

Introduce now Chevalley groups. For more details see \cite{Steinberg1},
\cite{Chevalley1}, \cite{Vavilov3}, \cite{Demazure}, \cite{Borel1}, \cite{Carter1}, \cite{Vavilov1}.

Consider semisimple linear algebraic groups over algebraically
closed fields. These are precisely elementary Chevalley groups
$E_\pi(\Phi,K)$.

All these groups are defined in $\SL_n(K)$ as  common set of zeros of
polynomials of matrix entries $a_{ij}$ with integer coefficients
 It is clear now that multiplication and
taking inverse element are  defined by polynomials with integer
coefficients. Therefore, these polynomials can be considered as
polynomials over an arbitrary commutative ring with a unit. Let some
elementary Chevalley group $E$ over~$\mathbb C$ be defined in
$\SL_n(\mathbb C)$ by polynomials $p_1(a_{ij}),\dots, p_m(a_{ij})$.
For a commutative ring~$R$ with a unit let us consider the group
$$
G(R)=\{ (a_{ij})\in \SL_n(R)\mid \widetilde p_1(a_{ij})=0,\dots
,\widetilde p_m(a_{ij})=0\},
$$
where  $\widetilde p_1(\dots),\dots \widetilde p_m(\dots)$ are
polynomials having the same coefficients as
$p_1(\dots),\dots,p_m(\dots)$, but considered over~$R$.

This group is called \emph{the Chevalley group} $G_{\ad}(\mathbf G_{2},R)$ of the type $\mathbf G_2$ over the ring~$R$, and it coincides with the elementary Chevalley group for every algebraically
closed field~$K$.

\bigskip

Define two  types of automorphisms of a Chevalley group $G_{\ad}(\mathbf G_2,R)$, we call them \emph{standard}.

{\bf Ring automorphisms.} Let $\rho: R\to R$ be an automorphism of
the ring~$R$. The mapping $(a_{i,j})\mapsto (\rho (a_{i,j}))$ from $G_{\ad}(\mathbf G_2,R)$
onto itself is an automorphism of the group $G_{\ad}(\mathbf G_2,R)$, 
denoted by the same letter~$\rho$. It is called a \emph{ring
automorphism} of the group~$G_{\ad}(\mathbf G_2,R)$. Note that for all
$\alpha\in \mathbf G_2$ and $t\in R$ an element $x_\alpha(t)$ is mapped to
$x_\alpha(\rho(t))$.

{\bf Inner automorphisms.} Let $S$ be some ring containing~$R$,  $g$
be an element of $G_{\ad}(\mathbf G_2,S)$, that normalizes the subgroup $G_{\ad}(\mathbf G_2,R)$. Then
the mapping $x\mapsto gxg^{-1}$  is an automorphism
of the group~$G_{\ad}(\mathbf G_2,R)$, denoted by $i_g$.  It is called an
\emph{inner automorphism}, \emph{induced by the element}~$g\in G_{\ad}(\mathbf G_2,S)$. If $g\in G_{\ad}(\mathbf G_2,R)$, then we call $i_g$ a \emph{strictly inner}
automorphism.

These two automorphisms are called \emph{standard}. Also central and graph automorphisms are called standard,
but in the case under consideration nontrivial  central and graph  automorphisms do not exist, therefore we will call an automorphism of a group $G(R)$ \emph{standard}, if it is a composition of two introduced types of automorphisms.

Our goal is to prove the following theorem:

\begin{theorem}\label{main}
Let $G=G_{\ad}(\mathbf G_{2},R)$ or $G=E_{\ad}(\mathbf G_{2},R)$
be a Chevalley group of the type~$\mathbf G_2$ or its elementary subgroup,  $R$ be a commutative ring with~$1/3$,  generated by its invertible elements and the ideal~$2R$. Then any automorphism of the group~$G$
is standard, i.\,e., is a composition of a ring and strictly inner automorphisms.
\end{theorem}

\bigskip

\section{Notions and theorems, necessary for the proof.}

\subsection{Localization of rings and modules; injection of a ring into the product of its localizations.}\leavevmode

\begin{definition}  Let $R$ be a commutative ring. A subset $Y\subset R$ is called \emph{multiplicatively closed} in~$R$, if $1\in Y$ and $Y$ is closed under multiplication.
\end{definition}

Introduce  an equivalence relation $\sim$ on the set of pairs $R\times Y$ as follows:
$$
\frac{a}{s}\sim \frac{b}{t} \Longleftrightarrow \exists u\in Y:\ (at-bs)u=0.
$$
  By $\frac{a}{s}$ we denote the whole equivalence class of the pair $(a,s)$, by $Y^{-1}R$ we denote the set of all equivalence classes. On the set $S^{-1}R$ we can introduce the ring structure by
$$
\frac{a}{s}+\frac{b}{t}=\frac{at+bs}{st},\quad \frac{a}{s}\cdot \frac{b}{t}=\frac{ab}{st}.
$$

\begin{definition}
The ring $Y^{-1}R$ is called the \emph{ring of fractions of~$R$ with respect to~$Y$}.
\end{definition}

 Let $\mathfrak p$ be a prime ideal of~$R$. Then the set $Y=R\setminus {\mathfrak p}$ is multiplicatively closed (it is equivalent to the definition of the prime ideal). We will denote the ring of fractions  $Y^{-1}R$ in this case by $R_{\mathfrak p}$. The elements $\frac{a}{s}$, $a\in \mathfrak p$, form an ideal $\mathfrak M$ in~$R_{\mathfrak p}$. If $\frac{b}{t}\notin \mathfrak M$, then $b\in Y$, therefore $\frac{b}{t}$ is invertible in~$R_{\mathfrak p}$. Consequently the ideal $\mathfrak M$ consists of all non-invertible elements of the ring~$R_{\mathfrak p}$, i.\,e., $\mathfrak M$ is the greatest ideal of this ring, so $R_{\mathfrak p}$ is a local ring.

The process of passing from~$R$ to~$R_{\mathfrak p}$ is called  \emph{localization at~${\mathfrak p}$.}

The construction $S^{-1}A$ can be easily carried trough  with an  $A$-module~$M$.
 Let $m/s$ denote the equivalence class of the pair $(m,s)$, the set $S^{-1}M$ of all such fractions    is made as a module $S^{-1}M$ with obvious operations of addition and scalar multiplication. As above we will write $M_{\mathfrak p}$ instead of $S^{-1}M$ for $S=A\setminus {\mathfrak p}$, where $\mathfrak p$ is a prime ideal of~$A$.

\begin{proposition}\label{inlocal} 
Every commutative ring  $R$ with $1$ can be naturally embedded in the cartesian product of all its localizations  by maximal ideals
$$
S=\prod\limits_{{\mathfrak m}\text{  is a maximal ideal of }A} A_{\mathfrak m}
$$
by diagonal mapping, which corresponds every $a\in R$ to the element
$$
\prod\limits_{\mathfrak m} \left( \frac{a}{1}\right)_{\mathfrak m}
$$
 of the ring~$S$.
\end{proposition}

\subsection{Isomorphisms of Chevalley groups over fields.}\leavevmode

We will need the description of isomorphisms between Chevalley groups over fields.

\begin{theorem}[see Theorems~30  and~31 from \cite{Steinberg1}]\label{isom_fields}
 Let $G$, $G'$ be  Chevalley groups, constructed  with root system $\mathbf G_2$ and fields $k,k'$, respectively. Suppose that $k,k'$ have characteristics not equal to three,  $\varphi: G\to G'$ be a group isomorphism. Then  the fields $k$ and $k'$ are isomorphic, and the isomorphism $\varphi$ is a composition of a ring isomorphism between $G$ and $G'$, and  inner automorphism of~$G'$.
\end{theorem}

\subsection{The subgroup $E_{\ad}(\mathbf G_2,R)$ is characteristic in the $G_{\ad}(\mathbf G_2,R)$}\leavevmode

\begin{definition}
A subgroup $H$ of the group  $G$ is called \emph{characteristic}, if it it is mapped into itself under any automorphism of group~$G$.
In particular, any characteristic subgroup is normal.
\end{definition}

\begin{theorem}[\cite{Vaserstein}]\label{character}
If the rank of  an indecomposable  root system~$\Phi$ is more than one, then  the elementary group $E_{\ad}(G_{2},R)$ is characteristic in the Chevalley group $G_{\ad}(G_{2},R)$.
\end{theorem}

\subsection{Normal structure of Chevalley groups over commutative rings.}\leavevmode

If $R$ is a ring, $I$  is its ideal, then by $\lambda_I: G_{\ad}(\mathbf G_{2},R)\to G_{\ad}(\mathbf G_{2},R/I)$ ($E_{\ad}(\mathbf G_{2},R)\to E_{\ad}(\mathbf G_{2},R/I)$) we denote a homomorphism, obtained by corresponding  every matrix $A\in G_{\ad}(\mathbf G_{2},R)$ to its image under the natural homomorphism $R\to R/I$.

We denote the inverse image of the center of $G_{\ad}(R/I)$ under   the homomorphism  $\lambda_I$ by $Z_I$

\begin{theorem} (\cite{Abe6}) If a subgroup ${\calH}$ of $E_{\ad}(\mathbf G_2,R)$ is normal in $E_{\ad}(\mathbf G_2,R)$ then
$$
E_{\ad}(\mathbf G_2,R,I) \leq {\calH} \leq Z_{\ad}(\mathbf G_2,R,I) \cap E_{\ad}(\mathbf G_2,R) 
$$
for some uniquely defined ideal $I$ of the ring $R$.
\end{theorem}

Note, that the center of Chevalley group of type $\mathbf G_2$ is trivial.

\begin{definition}
Let  $N_I$ denotes the subgroup  $\kerr \lambda_I\cap E(\mathbf G_{2},R)$.
 \end{definition}

\begin{proposition}\label{p4_1}(see \cite{Bunina7})
Let $\varphi$ be an arbitrary automorphism of a group $E(\mathbf G_{2},R)$, $I$ be a maximal ideal of the ring~$R$. Then there exists a maximal ideal  $J$ of the ring~$R$ such that $\varphi (N_I)=N_J$.
\end{proposition}

\section{Formulation of main steps of the proof.}\leavevmode

Consider a ring $R$ and its maximal ideal~$I$. We denote the localization  $R$ with~$I$ by~$R_I$, and its radical  by ${\Rad} R_I$.  Note, that there are two isomorphic fields:  $R/I$ and $R_I/{\Rad} R_I$:
$$
\begin{CD}
R @>>> R_I\\
@V\lambda_I VV @VV \lambda_{\Rad R_I}V\\
R/I @>\mu_I >> R_I/\Rad R_I
\end{CD}
$$

Let now  $\varphi$  be an arbitrary automorphism of $E_{\ad}(\mathbf G_2,R)$. Proposition  \ref{p4_1} gives us a possibility to consider the commutative diagram
 \begin{equation}\label{diagramm}
   \begin{CD}
   E_{\ad}(\mathbf G_2,R) @> \varphi >>  E_{\ad}(\mathbf G_2,R)\\
   @V r_I VV @VVr_J V\\
   E_{\ad}(\mathbf G_2,R_I) @. E_{\ad}(\mathbf G_2,R_J)\\
   @V \lambda_{\Rad R_I}VV @VV \lambda_{\Rad R_J}V\\
   E_{\ad}(\mathbf G_2,R_I/\Rad R_I) @. E_{\ad}(\mathbf G_2,R_J /\Rad R_J)\\
   @V s_I VV @VV s_J V \\
   E_{\ad} (\mathbf G_2,R/I) @> \overline \varphi >> E_{\ad}(\mathbf G_2,R/J)
   \end{CD}
   \end{equation}

The groups $E_{\ad}(\mathbf G_{2},R/I)$ and $E_{\ad}(\mathbf G_{2},R/J)$  are just elementary Chevalley groups over fields, their isomorphism were described in  Theorem~\ref{isom_fields}.

Recall, that the  fields $R/I$ and $R/J$ are isomorphic (as above let us  denote the corresponding isomorphism by~$\rho$),  and 
$$
\overline \varphi (A) = i_g \circ \rho(A)\ \forall A\in E_{\ad}(G_{2},R/I),\quad  g\in G_{\ad}(\mathbf G_2, R/J)
$$

The description of automorphisms of $E(\mathbf G_{2},R)$ can be done by the following scheme. the ring $R$ is embedded into the ring  $S=\prod R_I$: the Cartesian product of all local rings~$R_I$, obtained by localization of~$R$ by different maximal ideals~$I$.

Clear that $E(\mathbf G_{2},R)$ is embedded into
    $$
    G(\mathbf G_{2},S)=G(\mathbf G_{2},\prod R_I).
    $$

{\bf First step.}
    
    We prove, that for every maximal ideal~$J$  we have
    $$
    r_J  \varphi (x_\alpha(1))=i_{g_J} r_J(x_\alpha(1)) ,
    $$
    where $g_J\in G_{\ad}(\mathbf G_2,\overline{R_J})$ (an extesion of $R_J$).

\bigskip

{\bf Second step.}

We prove, that an inner automorphism of the group $G_{\ad}(\mathbf G_2,S)$, generated by $g=\prod g_J$, induces an automorphism of the group~$G_{\ad}(\mathbf G_2,R)$ and in fact  is  strictly inner.

Further, we show, that if we take a composition of the initial automorphism and the inner automorphism $i_{g^{-1}}$, the obtained automorphism is a ring one.

\bigskip

\section{Proof of the first step.}\leavevmode

Consider an arbitrary element $x_\alpha(1)\in E_{\ad} (\mathbf G_2, R)$, $\alpha\in \mathbf G_2$. Its image under the mapping $r_I$  is  $x_\alpha (1)=x_\alpha (1/1)\in E_{\ad} (\mathbf G_2, R_I)$. In the field $R/I$ its image has the same form. The element $x_\alpha'=\varphi(x_\alpha(1))\in E_{\ad} (\mathbf G_2, R)$ under the factorization by the ideal~$J$ gives $\overline \varphi(x_\alpha(1))=i_{\overline g}  (x_\alpha (1))$, where $\overline g\in G_{\ad}(\mathbf G_2,R/J)$.

Now choose $g\in G_{\ad}(\mathbf G_2,R_J)$ such that under the factorization of ring $R_J$  by its radical  $g$ is mapped to $\overline g$.

Then consider the following mapping $\psi: E_{\ad} (\mathbf G_{2}, R)\to E_{\ad} (\mathbf G_{2}, R_J)$:
$$
\psi = i_{g^{-1}} \circ r_J\circ \varphi.
$$
Under the mapping $\psi$ all $x_\alpha (1)$, $\alpha\in \mathbf G_{2}$ are mapped to $x_\alpha'$ such that  $x_\alpha(1)-x_\alpha'\in M_N(\Rad R_J)$.

Therefore, we now have a set of elements $\{ x_\alpha' \mid \alpha \in \mathbf G_{2}\}\subset E_{\ad} (\mathbf G_{2},R_J)$, satisfying the same relations as $\{ x_\alpha (1)\mid \alpha \in \mathbf G_{2}\}$, and equivalent to the corresponding $x_\alpha(1)$  modulo radical~$R_J$.

It is precisely the situation of the paper~\cite{Bunina8}, in which for a local ring~$R$  and the root system $\mathbf G_2$ if $3\in R^*$ without any additional conditions for the ring it is proved, that  if in the group $E_{\ad} (\mathbf G_{2}, R)$ some elements $x_\alpha'$ are the images  of the corresponding  $x_\alpha(1)$, $\alpha\in \mathbf G_{2}$, and also $x_\alpha(1)-x_\alpha'\in M_N(\Rad R)$, then there exists  $g'\in G_{\ad} (\mathbf G_{2}, R)$ , $g'-E\in M_N(\Rad R)$ such that for any $\alpha\in \mathbf G_{2}$
$$
x_\alpha(1)= i_{g'} (x_\alpha').
$$

Thus, the first step of the theorem completely follows from the statements of the previous paragraph. $\square$

Now embedding the initial ring  $R$ into the ring $S=\prod\limits_J R_J$, we see that
$$
\varphi(x_\alpha(1))=g (x_\alpha(1)) g^{-1},
$$
where $g=\prod\limits_J g_J$.

\bigskip

\section{Proof of the second step of the theorem.}\leavevmode

Conjugation by an element $g \in G_{\ad} (\mathbf G_2, S)$ can be extended up to the automorphism of the whole matrix ring $M_N (S)$.

Let us prove the following lemma:

\begin{lemma}\label{generate} 
Any matrix $x_{\alpha}(t),\alpha \in \mathbf G_2$, $t\in R^*$ or $t\in 2R$, under  conjugation by $g$ is mapped into a matrix from $\GL_{N}(R)$.
\end{lemma}

\begin{proof}
For $x_{\alpha}(1)$,  where $ \alpha \in \mathbf G_{2}$,  this statement is true (see the first step).

Let $\alpha$ --- be a long root, $\alpha \in \mathbf G_2$. For the long roots let us consider
$$
x_{\alpha}(1) = E + X_{\alpha} + \frac{X_{\alpha}^2}{2}, \quad
$$
where
$
\frac{X_{\alpha}^2}{2} = E_{(\alpha,-\alpha)}.
$
Lets us take two long roots  $\gamma$, $\beta \in \mathbf G_{2}$  such that $\gamma + \beta = \alpha$.
Note that 
$$
(x_{\gamma}(1)x_{\beta}(1)-x_{\gamma}(1)-x_{\beta}(1)+E)^2 = \pm E_{(\alpha,-\alpha)},
$$
it easily gives~$X_{\alpha}$.

Then any matrix  $x_\alpha (t)$ under the conjugation by~$g$ is mapped into  a matrix from $M_{N}(R)$. 

Now consider some invertible  element from the ring~$R$. 
For any invertible element $t \in R^*$ 
$$
h_{\alpha}(t) = x_{\alpha}(t)x_{-\alpha}(-t^{-1})x_{\alpha}(t)x_{\alpha}(1)^{-1}x_{-\alpha}(-1)^{-1}x_{\alpha}(1)^{-1}.
$$
 Therefore,  conjugation~$g$  maps the element $h_{\alpha}(t)$  to a matrix from $M_{N}(R)$.

Now let us consider an arbitrary short root $\beta$ and a long root~$\alpha$, such that
$$
\langle \beta, \alpha \rangle =1.
$$
Then
$$
x_{\beta}(t^{\langle \beta, \alpha \rangle}) = h_{\alpha}(t)x_{\beta}(1) h_{\alpha}(t)^{-1},
$$
i.\,e., $x_{\beta}(t)$ for any invertible element~$t\in R^*$ is mapped under  conjugation by~$g$  into a matrix with coefficients from~$R$.

Now let again $\beta$ be some short root. Then
$$
x_{\beta}(\pm 1) = E \pm X_{\beta}+\frac{X_{\beta}^2}{2}  \pm \frac{X_{\beta}^3}{6},
$$
hence
$$
x_{\beta}(-1)+x_{\beta}(1) = 2E+X_{\beta}^2 
\Longrightarrow X_{\beta}^2 = x_{\beta}(-1)+x_{\beta}(1)-2E. \eqno (1) 
$$
Similarly
$$
x_{\beta}(1)-x_{\beta}(-1)= 2X_{\beta}+\frac{X_{\beta}^3}{3}
\Longrightarrow \frac{X_{\beta}^3}{3} = x_{\beta}(1) - x_{\beta}(-1) - 2X_{\beta}. \eqno (2)
$$
Besides,
\begin{multline*}
x_{\beta}(2) = E + 2X_{\beta} + 2X_{\beta}^2 + \frac{4X_{\beta}^3}{3} =\\
= E + 2X_{\beta} + 2x_{\beta}(1)+2x_{\beta}(-1) - 4E + 4x_{\beta}(1) - 4x_{\beta}(-1) - 8X_{\beta} = \\ =-3E-6X_{\beta}+6x_{\beta}(1)-2x_{\beta}(-1),
\end{multline*}
therefore, using the invertibility of $3$, we obtain
$$
 2X_{\beta} = \frac{-x_{\beta}(2)-2x_{\beta}(-1)}{3}-E+2x_{\beta}(1). \eqno (3)
$$
From (2) and (3) it now follows
$$
\frac{X_{\beta}^3}{3} = -x_{\beta}(1)-\frac{x_{\beta}(-1)}{3}+\frac{x_{\beta}(2)}{3}+E.        \eqno (4)
$$
Then any $x_\beta(t)$, where $t\in 2R$, can be expressed through  $x_{\beta}(1)$ and~$t$, therefore, it also mapped into a matrix with coefficients from~$R$ under the conjugation by~$g$.

Therefore all $x_\beta(t)$, where $t$ is  either invertible, or is divided by~$2$, are mapped into  matrices with  coefficients from~$R$ under the conjugation  by~$g$.
\end{proof}

\bigskip

If (as in our case)  $R$ is generated by invertible elements and the ideal~$2R$, then 
$$
g G_{\ad} (\mathbf G_2, R) g^{-1}\subseteq M_N (R)\cap G_{\ad} (\mathbf G_2, S) =G_{\ad} (\mathbf G_2, R),
$$
Therefore
$$
g G_{\ad} (\mathbf G_2, R) g^{-1}=G_{\ad} (\mathbf G_2, R) \; 
 \text{and}  \; g E_{\ad} (\mathbf G_2, R) g^{-1}=E_{\ad} (\mathbf G_2, R).
$$

Consequently,  if we  take  composition of the initial automorphism $\varphi\in \Aut (G_{\ad}(\mathbf G_2,R))$ and conjugation $i_{g^{-1}}$ by the element $g^{-1}$,  we get some automorphism $\rho \in \Aut (G_{\ad}(\mathbf G_2,R))$, for which $\rho(x_{\alpha}(1))=x_{\alpha}(1)$ for any $\alpha\in \mathbf G_2$.

\begin{lemma}\label{ringaut}
Any automorphism $\rho\in \Aut (G_{\ad}(\mathbf G_2,R))$ (or $\rho\in \Aut (E_{\ad}(\mathbf G_2,R))$), for which  for all $\alpha \in \mathbf G_2$ $\rho (x_\alpha(1))=x_\alpha(1)$, is a ring automorphism of the Chevalley group (respectively, its elementary subgroup).
\end{lemma}

\begin{proof}
For a given $\alpha\in \mathbf G_{2}$   by~$\Gamma_\alpha$  let us denote the set of all   $x_\beta(1)$ such that $[x_\alpha(1),x_\beta(1)]=e$. In the paper~\cite{Bunina_very_new} it is proved, that the centralizer of    $\Gamma_\alpha$ in the Chevalley group (or in its elementary subgroup) coincides with  $CX_\alpha$, where $C$ is the center of the Chevalley group, and  $X_\alpha=\{ x_\alpha(t)\mid t\in R\}$. In our root system $\mathbf G_{2}$ center is trivial (since the group is adjoint), so we get~$X_\alpha$.

If under the action of~$\rho$ the elements $x_\beta(1)$ are mapped into itself, so the centralizer  of any its set also is mapped into itself, so for any $t\in R$ there exists $s\in R$ such that $\rho(x_\alpha(t))=x_\alpha(s)$.

Let us show, that the obtained mapping $t\mapsto s$ does not depend of choosing a root~$\alpha\in \mathbf G_2$. Actually, it must coincide on  roots of the same length, since if $\alpha_1$ and $\alpha_2$ have the same length, there  exists an element from the Weil group  $w$ (generated only by $x_\beta(1)$) such that
$$
\forall t\in R\quad  wx_{\alpha_1}(t) w^{-1}= x_{\alpha_2}(t).
$$
Let us denote the obtained mapping  on long roots by~$\rho_1$, and on short roots by~$\rho_2$.  

Apply the automorphism $\rho$ to the expression $[x_\alpha(t), x_\beta(1)]$, where $\alpha,\beta$ are short roots of the system~$\mathbf G_2$, $\alpha$ is long, $\beta$ is short:
\begin{multline*}
\rho([x_\alpha(t), x_\beta(1)])=[\rho(x_\alpha(t)), \rho(x_\beta(1))]\Longleftrightarrow \\
\Longleftrightarrow x_{\alpha+\beta}(\rho_2(t))x_{\alpha+3\beta}(\rho_1(-t))x_{\alpha+2\beta}(\rho_2(-t))x_{2\alpha+3\beta}(\rho_1(t^2))=\\
=x_{\alpha+\beta}(\rho_1(t))x_{\alpha+3\beta}(-\rho_1(t)) x_{\alpha+2\beta}(-\rho_1(t))x_{2\alpha+3\beta}(\rho_1(t^2)),
\end{multline*}
so  $\rho_1$ and $\rho_2$ coincide.

Denote the obtained mapping also by~$\rho: R\to R$, we just need to prove that this mapping is an automorphism of~$R$.

Actually, its bijectivity follows from the fact  that the initial automorphism  $\rho$  is bijective.

Its additivity follows from the formula
\begin{multline*}
x_{\alpha}(\rho(t_1)+\rho(t_2))=x_\alpha(\rho(t_1))x_\alpha(\rho(t_2))=\rho(x_\alpha(t_1))\cdot \rho   (x_\alpha(t_2))=\\
=\rho(x_\alpha(t_1)x_\alpha(t_2))=\rho(x_\alpha(t_1+t_2)),
\end{multline*}
and multiplicativity follows from formula $(R2)$
\begin{multline*}
x_{2\alpha+3\beta}(\rho(t_1)\rho(t_2))=[x_{\alpha}(\rho(t_1)), x_{\alpha+3\beta}(\rho(t_2))]=\rho([x_{\alpha}(t_1),
x_{\alpha+3\beta}(t_2)])= \\
=\rho(x_{2\alpha+3\beta}(t_1t_2))=x_{2\alpha+3\beta}(\rho(t_1t_2))
\end{multline*}
where $\alpha,\beta$ are simple roots of $\mathbf G_2$.

Thus, $\rho$ is a ring automorphism on all $x_\alpha(t)$, $\alpha \in \mathbf G_2$, $t\in R$, therefore it is a ring automorphism of the elementary Chevalley group $E_{\ad} (\mathbf G_2, R)$. 
If an initial automorphism was considered only on the elementary subgroup, lemma is proved.

Then  consider the general Chevalley group $G_{\ad}(\mathbf G_2,R)$. If an automorphism $\rho$ coincides with ring automorphism on an elementary subgroup, so if we take the composition of it and the inverse mapping of the ring automorphism, we obtain the  automorphism $\rho'$ of the Chevalley group,  which is identical on its elementary subgroup. In this case for any $g\in G_{\ad}(\mathbf G_2,R)$ and all $x\in E_{\ad}(\mathbf G_2,R)$ we have
$$
gxg^{-1}=x' \Longrightarrow \rho'(g) x \rho'(g)^{-1}=x' \Longrightarrow \rho'(g)x\rho'(g)^{-1}=gxg^{-1} \Longrightarrow (g^{-1}\rho'(g)) x= x (g^{-1}\rho'(g)).
$$
As this equality is true for $x\in E_{\ad}(\mathbf G_2,R)$, and the centralizer of the elementary subroup of an adjoint Chevalley group is trivial, then  $g^{-1}\rho'(g)=e$ and $\rho'$  is a trivial automorphism.

Therefore, $\rho$  is a ring automorphism of the Chevalley group, and the initial automorphism $\varphi$ is a composition of a ring and inner automorphisms, what was required. 

\end{proof}

\end{document}